\documentclass[12pt]{amsart}

\usepackage{amssymb}

\usepackage{amsthm}
\usepackage{mathtools}
\usepackage{latexsym}
\usepackage{amsmath}
\usepackage{amssymb}

\usepackage{amsthm}
\usepackage[all]{xy}

\usepackage{hyperref}

\hypersetup{
    colorlinks=true,       
    linkcolor=blue,          
    citecolor=blue,        
    filecolor=blue,      
    urlcolor=blue           
}

\newtheorem{thm}{Theorem}
\newtheorem{prop}[thm]{Proposition}
\newtheorem{cor}[thm]{Corollary}
\newtheorem{lem}[thm]{Lemma}
\newtheorem{defn}[thm]{Definition}
\newtheorem{rem}[thm]{Remark}

\begin{document}

\title{On two upper bounds for hypersurfaces involving a Thas' invariant}

\author{Andrea Luigi Tironi}

\date{\today}

\address{
Departamento de Matem\'atica,
Universidad de Concepci\'on,
Casilla 160-C,
Concepci\'on, Chile}
\email{atironi@udec.cl}

\thanks{During the preparation of this paper, the author was partially supported 
by Proyecto VRID N. 214.013.039-1.OIN and the Project Anillo ACT 1415 PIA CONICYT}

\maketitle

\begin{abstract}
Let $X^n$ be a hypersurface in $\mathbb{P}^{n+1}$ with $n\geq 1$
defined over a finite field $\mathbb{F}_q$ of $q$ elements. In this note, we classify, up to projective equivalence, hypersurfaces $X^n$ as above which reach two elementary upper bounds for the number of $\mathbb{F}_q$-points on $X^n$ which involve a Thas' invariant.
\end{abstract}

\section{Introduction}\label{1}

Let $\mathbb{F}_q$ be a field of $q$ elements, where $q=p^r$ for some prime $p$ and some positive integer $r$, and let $X^n$ be a hypersurface in $\mathbb{P}^{n+1}$ defined over $\mathbb{F}_q$ of degree $d\geq 2$ and dimension $n\geq 1$. 
Several years ago, Thas defined in \cite{Th} an invariant $k_{X^n}$ of $X^n$, that is, the maximum dimension $k_{X^n}$ of an $\mathbb{F}_q$-linear subspace of $\mathbb{P}^{n+1}$ which is contained in $X^n$, and obtained an upper bound for the number $N_q(X)$ of $\mathbb{F}_q$-points of $X^n$ which involved this invariant $k_{X^n}$. Recently, Homma and Kim established the following elementary upper bound involving $k_{X^n}$ (cf. \cite[Theorem 3.2]{HK7}),
\begin{equation}\label{*}
N_q(X^n)\leq (d-1)q^{k_{X^n}}N_q(\mathbb{P}^{n-k_{X^n}})+N_q(\mathbb{P}^{k_{X^n}})\ ,
\end{equation}
which works well for $k_{X^n}>0$. Moreover, they proved that \eqref{*}
is better than Thas' upper bound (see, \cite[$\S 7.1$]{HK7}). Finally, 
in \cite{HK7} the authors gave 
the complete list of nonsingular hypersurfaces $X^n$ in $\mathbb{P}^{n+1}$ with $n$ even which reach the equality in
\eqref{*} for $k_{X^n}=\frac{n}{2}$ (see \cite[Theorem 4.1]{HK7}).

The main purpose of this article is to re-prove in an easy way the Homma-Kim's elementary upper bound \eqref{*} for $k_{X^n}>0$, extending this also to the case $k_{X^n}=0$, and to give a complete list of hypersurfaces $X^n$ in $\mathbb{P}^{n+1}$ which reach this bound, independently of the parity of $n$ and the singularities of $X^n$. In particular, observe that $k_{X^n}\leq n$ and note that 
the right hand of the inequality in \eqref{*} increases with $k_{X^n}$.
Thus, the upper bound 
in \eqref{*} reduces to the Segre-Serre-S{\o}rensen's upper bound (\cite{Seg}, \cite{Ser} and \cite{Sor}) for the general case $k_{X^n}\leq n$, and it becomes the Homma-Kim's elementary bound proved in \cite{HK4} for hypersurfaces $X^n$ which does not admit $\mathbb{F}_q$-linear components, that is, when $k_{X^n}\leq n-1$. Furthermore, in both of the above cases, a complete list of hypersurfaces $X^n$ in $\mathbb{P}^{n+1}$ 
achieving the upper bound in \eqref{*}
with $k_{X^n}=n,n-1$ is given in \cite{Ser} and \cite{T}, respectively.

Therefore, keeping in mind the two above cases, for $0< k_{X^n}\leq n$  we obtain the following classification result.

\begin{thm}[Cases $0<k_{X^n}\leq n$]\label{thm1}
Let $X^n\subset\mathbb{P}^{n+1}$ be a hypersurface of degree $d\geq 2$ and dimension $n\geq 1$ defined over $\mathbb{F}_q$. Define
$$k_{X^n}:=\max\left\{h\ |\ \mathrm{there\ exists\ an\ } \mathbb{F}_q\mathrm{-linear\ space\ } \mathbb{P}^h\subseteq X\right\}$$
and suppose that $0<k_{X^n}\leq n$. Then
$$N_q(X^n)\leq (d-1)q^{k_{X^n}}N_q(\mathbb{P}^{n-k_{X^n}})+N_q(\mathbb{P}^{k_{X^n}})$$
and equality holds if and only if one of the following possibilities occurs:
\begin{enumerate}
\item[(I)] $k_{X^n}=n$ and $X^n$ is a union of $d$ hyperplanes over $\mathbb{F}_q$ that contain a common $\mathbb{F}_q$-linear subspace of codimension $2$ in $\mathbb{P}^{n+1}$;
\item[(II)] $0<k_{X^n}\leq n-1$ and one of the following cases can occur:
\begin{enumerate}
\item[$(1)$] $d=q+1$ and $X^n$ is a space-filling hypersurface
$$(X_0,\ldots,X_{n+1})\ A\ {}^t \! (X_0^q,\ldots,X_{n+1}^q)=0,$$ where $A=\left( a_{ij}\right)_{i,j=1,\ldots,n+2}$ is an
$(n+2)\times (n+2)$ matrix such that ${}^t \! A=-A$ and $a_{kk}=0$ for every $k=1,\ldots,n+2$; moreover, $X^n$ is nonsingular if and only if $\det A\neq 0 ;$
in particular, if $n$ is odd, then $X^n$ is singular\ ;
\item[$(2)$]
$d=\sqrt{q}+1$ and 
\begin{itemize}
\item[(a)] $n=2h$ with $h\in\mathbb{Z}_{\geq 1}$, $1\leq k_{X^{2h}}\leq 2h-1$ and one of the following two cases holds:
\begin{enumerate}
\item[(i)] if $\mathrm{Sing}(X^{2h})(\mathbb{F}_q)=\emptyset$, then $k_{X^{2h}}=h$ and $X^{2h}$ is projectively equivalent to a 
nonsingular Hermitian hypersurface;
\item[(ii)] if $\mathrm{Sing}(X^{2h})(\mathbb{F}_q)\neq\emptyset$, then $h\geq 2$ and, up to projective equivalence, we have
\abovedisplayskip=0pt\relax
\[
X^{2h} =
\begin{cases}
\mathbb{P}^1*X_H^{2h-2}\ , & k_{X^{2h}}=h+1 \\
\mathbb{P}^3*X_H^{2h-4}\ , & k_{X^{2h}}=h+2 \\
\dots & \\
\mathbb{P}^{2h-3}*X_H^{2}\ , & k_{X^{2h}}=h+(h-1)\ ; \\
\end{cases}
\]
\end{enumerate}
\item[(b)] $n=2h+1$ with $h\in\mathbb{Z}_{\geq 1}$, $1\leq k_{X^{2h+1}}\leq 2h$
and, up to projective equivalence, we have
\abovedisplayskip=0pt\relax
\[
X^{2h+1} = 
\begin{cases}
\mathbb{P}^0*X_H^{2h}\ , & k_{X^{2h+1}}=h+1 \\
\mathbb{P}^2*X_H^{2h-2}\ , & k_{X^{2h+1}}=h+2 \\
\mathbb{P}^4*X_H^{2h-4}\ , & k_{X^{2h+1}}=h+3 \\
\dots & \\
\mathbb{P}^{2h-2}*X_H^{2}\ , & k_{X^{2h+1}}=h+h\ , \\
\end{cases}
\]
\end{itemize}
where $\mathbb{P}^l*X_H^{m}\subset\mathbb{P}^{m+l+2}$ is a cone over a nonsingular Hermitian 
$\mathbb{F}_q$-hypersurface $X_H^{m}\subset\mathbb{P}^{m+1}$ of dimension $m$ with vertex an $\mathbb{F}_q$-linear subspace $\mathbb{P}^l$\ ;

\item[$(3)$]
$d=2$, $k_{X^n}=\frac{n+h+1}{2}\in\mathbb{Z}_{>0}$ and $X^n$ is projectively equivalent to a cone $\mathbb{P}^h*Q^{n-h-1}\subset\mathbb{P}^{n+1}$ with vertex an
$\mathbb{F}_q$-linear subspace $\mathbb{P}^h$ with $-1\leq h\leq n-1$, where $Q^{n-h-1}\subset \mathbb{P}^{n-h}$ is the hyperbolic quadric hypersurface 
$$X_0X_1+X_2X_3+ \dots + X_{n-h-1}X_{n-h}=0\ .$$
\end{enumerate}
\end{enumerate}
\end{thm}

As to the case $k_{X^n}=0$, let us recall here that Homma obtained in \cite{H} an upper bound for hypersurfaces $X^n\subset\mathbb{P}^{n+1}$ with $n\geq 1$ without $\mathbb{F}_q$-lines which works well except for the case $n=1$ and $d=q=4$. On the other hand, his bound is better than \eqref{*} with $k_{X^n}=0$. For these reasons, we provide here another elementary upper bound for the number of $\mathbb{F}_q$-points of hypersurfaces $X^n$ in $\mathbb{P}^{n+1}$ with $k_{X^n}=0$ for any $n\geq 1$ and we characterize those $X^n$ which achieve this bound in the following result. 

\begin{thm}[Case $k_{X^n}=0$]\label{thm2}
Let $X^n\subset\mathbb{P}^{n+1}$ be a hypersurface of degree $d\geq 2$ and dimension $n\geq 1$ defined over $\mathbb{F}_q$. If $k_{X^n}=0$, then
$$N_q(X^n)\leq (d-1)q^n+(d-2)N_q(\mathbb{P}^{n-1})+1$$ 
and equality holds if and only if $d=2$ and, up to projective equivalence, either $n=1$ and $X^1:\ X_0^2+X_1^2+X_2^2=0$ is a nonsingular plane conic, or
$n=2$ and $X^2:\ f(X_0,X_1)+X_2X_3=0$ in a nonsingular elliptic surface, where $f(X_0,X_1)=\alpha X_0^2+X_0X_1+X_1^2$ is an irreducible binary quadratic form
with $\alpha\in\{ t\in\mathbb{F}_q\ |\ t+t^2+t^4+ \dots +t^{2^{r-1}}=1 \}$ if $q=2^r$ for some $r\in\mathbb{Z}_{\geq 1}$ and 
such that $1-4\alpha $ is a non-square if $q$ is odd. 
\end{thm}

Finally, in Corollary \ref{corollary} of Section \ref{sec2} we give an immediate consequence of Theorems \ref{thm1} and \ref{thm2} for the nonsingular case.

\section{Notation and preliminary results}\label{sec1}

Let $X^n\subset\mathbb{P}^{n+1}$ be a hypersurface of degree $d\geq
2$ and dimension $n\geq 1$ defined over a finite field
$\mathbb{F}_q$ of $q$ elements, with $q=p^r$ for some prime number
$p$ and an integer $r\in\mathbb{Z}_{\geq 1}$. If $Y$ is an algebraic set in $\mathbb{P}^{n+1}$
defined by equations over $\mathbb{F}_q$,
the set of $\mathbb{F}_q$-points of $Y$ is denoted by $Y(\mathbb{F}_q)$ and the cardinality of $Y(\mathbb{F}_q)$ by $N_q(Y)$. 
Moreover, if $L$ is an $\mathbb{F}_q$-linear subspace 
of $\mathbb{P}^{n+1}$, then $L^\nu$ will denote the set of all 
$\mathbb{F}_q$-linear subspaces $\mathbb{P}^{\dim L+1}\subseteq\mathbb{P}^{n+1}$ containing $L$. 
Recall that for any $N\in\mathbb{Z}_{\geq 1}$ we have
$$N_q(\mathbb{P}^{N})=q^N+q^{N-1}+\dots +q+1$$
and define $\mathbb{P}^{-1}=\emptyset$. Finally, let us denote here by $\mathbb{P}^h*Y$ with $h\in\mathbb{Z}_{\geq -1}$ the cone with vertex $\mathbb{P}^h$ over the variety $Y$.

In this section, we give some preliminary results
which will be useful in Section \ref{sec2} to prove Theorems \ref{thm1} and \ref{thm2}.

First of all, let us re-prove in an easier and immediate way the same inequality as in \cite[Theorem 3.2]{HK7}.

\begin{prop}\label{prop1}
Let $X^n\subset\mathbb{P}^{n+1}$ be a hypersurface of degree $d\geq 2$ and dimension $n\geq 1$ defined over $\mathbb{F}_q$. 
If $1\leq k_{X^n}\leq n$, then $$N_q(X^n)\leq (d-1)q^{k_{X^n}}N_q(\mathbb{P}^{n-k_{X^n}})+N_q(\mathbb{P}^{k_{X^n}})\ .$$
\end{prop}

\begin{proof}
Consider an $\mathbb{F}_q$-linear subspaces $L$ of dimension $k_{X^n}$ contained in $X^n$. Then by  \cite{Ser} (see also \cite{Seg} and \cite{Sor}), we have
{\small
\begin{eqnarray*}
N_q(X^n) & = & \sum_{L'\in L^\nu}\left[N_q(X^n\cap L')-N_q(L)\right] + N_q(L)\\
& \leq & \left[ \left( dq^{k_{X^n}}+q^{k_{X^n}-1}+ \dots +1\right) - N_q(L)\right] \cdot N_q(\mathbb{P}^{n-k_{X^n}})+N_q(L)\\
& = & (d-1)q^{k_{X^n}}\cdot N_q(\mathbb{P}^{n-k_{X^n}})+N_q(L)\ ,
\end{eqnarray*}}
where $N_q(L)=N_q(\mathbb{P}^{k_{X^n}})=q^{k_{X^n}} + \dots + q+1$.
\end{proof}

\begin{rem}\label{rem}
{\em Fixing $d,q$ and $n$, the upper bound in Proposition \ref{prop1} increases with $k_{X^n}$. Thus, since $k_{X^n}\leq n$, from Proposition \ref{prop1} we deduce 
immediately the Segre--Serre--S{\o}rensen bound (see, \cite{Seg}, \cite{Ser} and \cite{Sor}). Moreover, if $X^n$ in $\mathbb{P}^{n+1}$ 
does not admit $\mathbb{F}_q$-linear components, then $k_{X^n}\leq n-1$ and Proposition \ref{prop1} gives the elementary Homma--Kim bound
(cf. \cite[Theorem 1.2]{HK4} and \cite[Remark 3.3]{HK7}).}
\end{rem}

As to the case $k_{X^n}=0$, i.e. when $X^n\subset\mathbb{P}^{n+1}$ is a hypersurface without $\mathbb{F}_q$-lines, 
with a technique different from the one used in Proposition \ref{prop1},
we can prove the following elementary upper bound.

\begin{prop}\label{prop2}
Let $X^n\subset\mathbb{P}^{n+1}$ be a hypersurface of degree $d\geq 2$ and dimension $n\geq 1$ defined over $\mathbb{F}_q$. 
Assume that $k_{X^n}=0$. Then $$N_q(X^n)\leq (d-1)q^n+(d-2)N_q(\mathbb{P}^{n-1})+1\ .$$
Moreover, if there exists a singular $\mathbb{F}_q$-point on $X^n$, then
$$N_q(X^n)\leq (d-2)q^n+(d-2)N_q(\mathbb{P}^{n-1})+1\ .$$
\end{prop}

\begin{proof}
Let $p\in X(\mathbb{F}_q)$. Take an $\mathbb{F}_q$-linear subspace $L=\mathbb{P}^n$ such that $p\notin L$
and consider the $\mathbb{F}_q$-linear tangent space $T_pX$ of $X^n$ at the point $p$. Note that $\mathrm{mult}_p(l\cap X^n)\geq 2$
for every $\mathbb{F}_q$-line $l\subset T_pX$ passing through the point $p$. Define $L':=T_pX\cap L$ and observe that
$L'$ is an $\mathbb{F}_q$-linear subspace of $\mathbb{P}^{n+1}$ of dimension $n-1, $ or $n$, depending on whether p is a 
nonsingular or singular point, respectively. Then we get
\begin{eqnarray*}
N_q(X^n) & = & \sum_{l\in p^\nu}\left[N_q(X\cap l)-N_q(p\cap l)\right] +N_q(p)\\
& \leq & \sum_{l\in p^\nu\ :\ l\cap L'\neq\emptyset}\left(d-2\right) + \sum_{l\in p^\nu\ :\ l\cap L'=\emptyset}\left(d-1\right)+1\\
& \leq & (d-2)N_q(L')+(d-1)\left[N_q(L)-N_q(L')\right]+1\ .
\end{eqnarray*}
Suppose that $p$ is nonsingular for $X^n$. Then $L'=\mathbb{P}^{n-1}$ and 
\begin{eqnarray*}
N_q(X^n) & \leq & (d-2)N_q(L')+(d-1)\left[N_q(L)-N_q(L')\right]+1\\
& = & (d-1)q^n+(d-2)N_q(\mathbb{P}^{n-1})+1\ .
\end{eqnarray*}
On the other hand, if $p$ is singular for $X^n$, then $L'=L=\mathbb{P}^{n}$. Hence
$N_q(X^n)\leq (d-2)N_q(L)+1=(d-2)q^n+(d-2)N_q(\mathbb{P}^{n-1})+1$.
\end{proof}

\medskip

The above results allow us to give the following definition.

\begin{defn}\label{def}
\abovedisplayskip=0pt\relax
\[
\Theta_{n,k_{X^n}}^{d,q} :=
\begin{cases}
(d-1)q^{k_{X^n}}N_q(\mathbb{P}^{n-k_{X^n}})+N_q(\mathbb{P}^{k_{X^n}}) & \text{if } 0<k_{X^n}\leq n\\
(d-1)q^n+(d-2)N_q(\mathbb{P}^{n-1})+1 & \text{if } k_{X^n}=0
\end{cases}
\]
\end{defn}

\begin{rem}
{\em We have $\Theta_{n,k_{X^n}}^{d,q}\leq N_q(\mathbb{P}^{n+1})$ if and only if $d\leq q+1$; moreover, if $k_{X^n}>0$ then equality holds if and only if $d=q+1$.}
\end{rem}

Denoting by $\mathrm{Sing}(X^n)$ the set of singular points of $X^n$, let us give here a technical result which will be useful to prove Theorem \ref{thm1} (see also \cite[$\S 5$]{HK7}).

\begin{lem}\label{lemma}
Let $X^n\subset\mathbb{P}^{n+1}$ be a hypersurface of degree $d\geq 2$ and dimension $n\geq 1$ defined over $\mathbb{F}_q$. 
Assume that $k_{X^n}>0$ and $N_q(X^n)=\Theta_{n,k_{X^n}}^{d,q}$. Then we have the following properties:
\begin{enumerate}
\item for any point $p\in X^n(\mathbb{F}_q)$ there exists at least an $\mathbb{F}_q$-linear subspace $\mathbb{P}^{k_{X^n}}$ such that
$p\in\mathbb{P}^{k_{X^n}}\subseteq X^n$;
\item if $p\in\mathrm{Sing}(X^n)(\mathbb{F}_q)$, then $p\in\mathbb{P}^{k_{X^n}}$ for any $\mathbb{F}_q$-linear subspace $\mathbb{P}^{k_{X^n}}\subseteq X^n$;
\item if $p\in\mathrm{Sing}(X^n)(\mathbb{F}_q)$ and $d=\sqrt{q}+1$, then $X^n=p*X^{n-1}$, that is, $X^n$ is a cone over an $\mathbb{F}_q$-subvariety
$X^{n-1}$ of dimension $n-1$ and degree $\sqrt{q}+1$;
\item if $p\in X^n(\mathbb{F}_q)$ is a nonsingular point and $0<k_{X^n}\leq n-1$, then $$N_q(X^n\cap T_pX^n)=\Theta_{n-1,k_{X^n}}^{d,q}\ .$$
\end{enumerate}
\end{lem}
\begin{proof}
$(1)$ Consider $p\in X^n(\mathbb{F}_q)$ and let $L=\mathbb{P}^{k_{X^n}}$ be an $\mathbb{F}_q$-linear subspace contained in $X^n$. If $p\in L$, then we are done. So, assume that $p\notin L$. Take an $\mathbb{F}_q$-linear subspace
$L'=\mathbb{P}^{k_{X^n}+1}$ such that $\{ p\}\cup\mathbb{P}^{k_{X^n}}\subset L'=\mathbb{P}^{k_{X^n}+1}$. Since $N_q(X^n)=\Theta_{n,k_{X^n}}^{d,q}$, from the 
proof of Proposition \ref{prop1} we deduce that $p\in\mathbb{P}^{k_{X^n}+1}\cap X^n=\cup_{i=1}^{d}\mathbb{P}_i^{k_{X^n}}$, i.e. $p\in\mathbb{P}_j^{k_{X^n}}\subseteq X^n$
for some $j\in\{1,\dots,d\}$.

\medskip

\noindent $(2)$ Let $p\in\mathrm{Sing}(X^n)(\mathbb{F}_q)$ and suppose that there exists an $\mathbb{F}_q$-linear subspace $\mathbb{P}^{k_{X^n}}\subseteq X^n$ which does not contain the point $p$.
Take $\mathbb{P}^{k_{X^n}+1}:=\langle p,\mathbb{P}^{k_{X^n}}\rangle$ the $\mathbb{F}_q$-linear subspace of $\mathbb{P}^{n+1}$ spanned by $p$ and $\mathbb{P}^{k_{X^n}}$. Note that $\mathbb{P}^{k_{X^n}+1}$ cannot be contained in $X^n$.
Since $N_q(X^n)=\Theta_{n,k_{X^n}}^{d,q}$, from the 
proof of Proposition \ref{prop1} we deduce that $p\in\mathbb{P}^{k_{X^n}+1}\cap X^n=\cup_{i=2}^{d}\mathbb{P}_i^{k_{X^n}}\cup\mathbb{P}^{k_{X^n}}$.
Thus $p\notin\mathbb{P}^{k_{X^n}}$ is not a singular point in $\mathbb{P}^{k_{X^n}+1}\cap X^n$ and by \cite[Lemma 2.6]{HK7} we conclude that $p$ is a nonsingular $\mathbb{F}_q$-point in $X^n$, 
but this is a contradiction. 

\medskip

\noindent $(3)$ Let $p\in\mathrm{Sing}(X^n)(\mathbb{F}_q)$ and consider an $\mathbb{F}_q$-linear subspace $L=\mathbb{P}^n\nsubseteq X^n$ which does not contain $p$.
Define $X^{n-1}:=X^n\cap L$. For any $q\in X^{n-1}(\mathbb{F}_q)\subseteq X^n(\mathbb{F}_q)$, by $(1)$ we see that there exists an $\mathbb{F}_q$-linear subspace
$\mathbb{P}^{k_{X^n}}\subset X^n$ such that $q\in\mathbb{P}^{k_{X^n}}$. Moreover, by $(2)$ we have also $p\in\mathbb{P}^{k_{X^n}}$ because $p\in\mathrm{Sing}(X^n)(\mathbb{F}_q)$.
Hence $\langle p,q\rangle\subseteq\mathbb{P}^{k_{X^n}}\subseteq X^n$. This shows that $p*X^{n-1}(\mathbb{F}_q)\subseteq X^n(\mathbb{F}_q)$. Now, let $p'\in X^n(\mathbb{F}_q)$ with $p'\neq p$.
From $(1)$ and $(2)$ we know that there exists an $\mathbb{F}_q$-linear subspace $\mathbb{P}^{k_{X^n}}$ such that $\langle p,p'\rangle\subseteq\mathbb{P}^{k_{X^n}}\subseteq X^n$.
Define $p'':=\langle p,p'\rangle\cap L$. Then $p''\in X^n\cap L=X^{n-1}$ and $p'\in\langle p,p''\rangle$. This gives $X^n(\mathbb{F}_q)\subseteq p*X^{n-1}(\mathbb{F}_q)$,
that is, $X^n(\mathbb{F}_q)=p*X^{n-1}(\mathbb{F}_q)$. Hence $N_q(X^n)=qN_q(X^{n-1})+1$ and this leads to
$$N_q(X^{n-1})=\frac{N_q(X^n)-1}{q}=\sqrt{q}q^{k_{X^n}-1}N_q(\mathbb{P}^{n-k_{X^n}})+N_q(\mathbb{P}^{k_{X^n}-1})\ .$$
Note that $k_{X^{n-1}}=k_{X^n}-1$, $\deg X^{n-1}=\deg X^n=\sqrt{q}+1$ and $\dim X^{n-1}=n-1$ by $(2)$ and the choice of $L$.
Thus
$$N_q(X^{n-1})>\left(\deg X^{n-1}-1\right)q^{n-1}+q^{n-2}+ \dots +q+1\ .$$
Since $N_q(X^{n-1})>\left(\deg X^{n-1}-1\right)q^{n-1}+N_q(\mathbb{P}^{n-2})$, $X^{n-1}\subseteq X^n$ and $(p*X^{n-1})(\mathbb{F}_q)\subseteq X^n$, from \cite[Proposition 2.8]{HK7} we conclude that 
$X^n=p*X^{n-1}$.

\medskip

\noindent $(4)$ Let $p\in X^n(\mathbb{F}_q)$ be a nonsingular point. Then by $(1)$ we know that there exists an $\mathbb{F}_q$-linear subspace $\mathbb{P}^{k_{X^n}}$
such that $p\in\mathbb{P}^{k_{X^n}}\subseteq X^n$. Thus $\mathbb{P}^{k_{X^n}}\subset T_pX^n$, where $T_pX^n=\mathbb{P}^{n}$ is the tangent $\mathbb{F}_q$-linear space of $X^n$ at $p$. 
Define $X^{n-1}:=X^n\cap T_pX^n$.
Hence
from the proof of Proposition \ref{prop1} it follows that
\begin{eqnarray*}
N_q(X^{n-1}) & = & \sum_{L\in\left({\mathbb{P}^{k_{X^n}}}\right)^\nu\ :\ L\subset T_pX^n}\left[N_q(X^n\cap L)-N_q(\mathbb{P}^{k_{X^n}})\right]+N_q(\mathbb{P}^{k_{X^n}})\\
& = & \left[dq^{k_{X^n}}+q^{k_{X^n}-1}+ \dots + 1 - N_q(\mathbb{P}^{k_{X^n}})\right]\cdot N_q(\mathbb{P}^{n-k_{X^n}-1})\\
& & \qquad +N_q(\mathbb{P}^{k_{X^n}})\\
& = & (d-1)q^{k_{X^n}}N_q(\mathbb{P}^{(n-1)-k_{X^n}})+N_q(\mathbb{P}^{k_{X^n}})\ .
\end{eqnarray*}

\noindent Since $k_{X^{n-1}}=k_{X^n}$, we obtain that $N_q(X^n\cap T_pX^n)=\Theta_{n-1,k_{X^n}}^{d,q}$.
\end{proof}

\section{Proof of Theorems \ref{thm1} and \ref{thm2}}\label{sec2}

In this section, by applying the previous results, we prove the two theorems stated in the Introduction. Finally, for the nonsingular case, an immediate consequence of them is given in Corollary \ref{corollary}. 

\bigskip

\noindent\textit{Proof of} Theorem \ref{thm2}. Assume that $k_{X^n}=0$ and note that the first part of the statement follows from Proposition \ref{prop2}. Thus, suppose that $N_q(X^n)=\Theta_{n,0}^{d,q}$. Then by \cite{H} we know that
$$\Theta_{n,0}^{d,q}=N_q(X^n)\leq (d-1)(q^n+1)+(d-2)\left(N_q(\mathbb{P}^{n-2})-1\right)\ .$$ This gives $(d-2)q^{n-1}\leq 0$, i.e. $d\leq 2$. Hence $d=2$, that is, $X^n$ is a quadric hypersurface, and $N_q(X^n)=\Theta_{n,0}^{2,q}=q^n+1$. Write $X^n:=\mathbb{P}^h*Q^{n-h-1}$, where $h\in\mathbb{Z}_{\geq -1}$ and $Q^{n-h-1}\subset\mathbb{P}^{n-h}$ is a nonsingular quadric hypersurface of dimension $n-h-1$. Then $N_q(X^n)=q^{h+1}N_q(Q^{n-h-1})+N_q(\mathbb{P}^h)$. Since up to projective equivalence $Q^{n-h-1}$ can be a parabolic, a hyperbolic or an elliptic quadric hypersurface, from \cite[Ch. $5$]{Hir} we deduce that
\abovedisplayskip=0pt\relax
\[
N_q(X^n)=
\begin{cases}
N_q(\mathbb{P}^{n}) & \text{if } n-h-1\ \text{is\ odd} \\ 
& \\
\frac{q^{h+1}\left(q^{\frac{n-h-1}{2}}\pm 1\right)\left(q^{\frac{n-h+1}{2}}\mp 1\right)}{q-1}+N_q(\mathbb{P}^h) & \text{if } n-h-1\ \text{is\ even}
\end{cases}
\]
By comparing the previous value $N_q(X^n)=q^n+1$ with the two above situations, we conclude that either $(n,h)=(1,-1)$ and $X^1$ is a nonsingular plane conic, or $(n,h)=(2,-1)$ and $X^2\subset\mathbb{P}^3$ is a nonsingular elliptic quadric surface. This concludes the proof of Theorem \ref{thm2}. \hfill $\square$

\bigskip

\noindent\textit{Proof of} Theorem \ref{thm1}. Suppose that $0<k_{X^n}\leq n$ and $N_q(X^n)=\Theta_{n,k_{X^n}}^{d,q}$. Note that the first part of the statement follows from Proposition \ref{prop1}. Moreover, if $k_{X^n}=n$ then by \cite{Ser} we can conclude. Thus, we can assume that $0<k_{X^n}\leq n-1$. From the proof of Proposition \ref{prop1} we deduce that there exists an $\mathbb{F}_q$-linear subspace $\overline{L}=\mathbb{P}^{k_{X^n}+1}\subset\mathbb{P}^{n+1}$ such that $X\cap\overline{L}=\cup_{i=1}^{d}\mathbb{P}_i^{k_{X^n}}$. By considering all the $\mathbb{P}^{k_{X^n}+2}$'s such that $\overline{L}\subset\mathbb{P}^{k_{X^n}+2}$ and $\mathbb{P}^{k_{X^n}+2}$ is an $\mathbb{F}_q$-linear subspace of $\mathbb{P}^{n+1}$, 
by \cite{HK4} we get
{\small
\begin{eqnarray*}
N_q(X^n) & = & \sum_{L\in\overline{L}^\nu}\left[ N_q(X^n\cap L)-N_q(X^n\cap\overline{L}) \right]+N_q(X^n\cap\overline{L})\\
& = & \left[\sum_{L\in\overline{L}^\nu} N_q(X^n\cap L)\right]-N_q(X^n\cap\overline{L})N_q(\mathbb{P}^{n-k_{X^n}-1})+N_q(X^n\cap\overline{L})\\
& \leq & \left[(d-1)q^{k_{X^n}+1}+dq^{k_{X^n}}+q^{k_{X^n}-1}+ \dots + 1\right]\cdot N_q(\mathbb{P}^{n-k_{X^n}-1})\\
& & \ -(dq^{k_{X^n}}+q^{k_{X^n}-1}+ \dots + 1)\left[N_q(\mathbb{P}^{n-k_{X^n}-1})-1\right]\\
& = & (d-1)q^{k_{X^n}}N_q(\mathbb{P}^{n-k_{X^n}})+N_q(\mathbb{P}^{k_{X^n}})\ ,
\end{eqnarray*}}

\noindent because $X\cap L\subset\mathbb{P}^{k_{X^n}+2}$ is an $\mathbb{F}_q$-hypersurface without linear $\mathbb{F}_q$-components for any $L\in\overline{L}^\nu$.
Since $N_q(X^n)=\Theta_{n,k_{X^n}}^{d,q}$, we see that $N_q(X\cap L)=(d-1)q^{k_{X^n}+1}+dq^{k_{X^n}}+q^{k_{X^n}-1}+ \dots + 1$ and from \cite{T} it follows that $d=\deg (X\cap L)\in\{2,\sqrt{q}+1,q+1\}$. We proceed now with a case-by-case analysis.

\medskip

Assume that $d=q+1$. Then by \cite[Proposition 14]{T} we know that 
$X^n$ is a space-filling hypersurface as in case $(1)$ of Theorem \ref{thm1}. 

\medskip

Suppose now that $d=2$. Write $X^n:=\mathbb{P}^h*Q^{n-h-1}$
for some $h\in\mathbb{Z}_{\geq -1}$, where
$Q^{n-h-1}\subset\mathbb{P}^{n-h}$ is a nonsingular quadric hypersurface. 
Note that 

\begin{equation}\label{d=2}
q^{h+1}N_q(Q^{n-h-1})+N_q(\mathbb{P}^h)=N_q(X^n)=q^{k_{X^n}}N_q(\mathbb{P}^{n-k_{X^n}})+N_q(\mathbb{P}^{k_{X^n}})\ .
\end{equation}

\noindent If $n-h$ is even, i.e. $n-h=2s$ for some $s\in\mathbb{Z}_{\geq 1}$, then $Q^{n-h-1}$ is a parabolic quadric hypersurface  
which contains $\mathbb{F}_q$-linear subspaces $\mathbb{P}^{\frac{n-h}{2}-1}$ of maximal dimension and such that $N_q(Q^{n-h-1})=N_q(\mathbb{P}^{n-h-1})$. Thus $k_{X^n}=\frac{n+h}{2}$ and by \eqref{d=2} we get
$$N_q(\mathbb{P}^n)=q^{\frac{n+h}{2}}N_q(\mathbb{P}^{n-\frac{n+h}{2}})+N_q(\mathbb{P}^{\frac{n+h}{2}})\ .$$
Hence $q^n+ \dots + 1=q^n+ \dots + q^{k_{X^n}+1} + 2q^{k_{X^n}} + q^{k_{X^n}-1} + \dots + 1$, but this gives a contradiction because $0< k_{X^n}\leq n-1$.
So, let $n-h=2s-1$ for some $s\in\mathbb{Z}_{\geq 1}$. Then $Q^{n-h-1}$ is either (i) a hyperbolic or (ii) an elliptic quadric hypersurface  
which contains $\mathbb{F}_q$-linear subspaces $\mathbb{P}^{m}$ of maximal dimension, where $m$ is either $\frac{n-h-1}{2}$ or $\frac{n-h-1}{2}-1$, respectively. Hence
we deduce that $k_{X^n}=\frac{n+h+1}{2}$ in case (i) and $k_{X^n}=\frac{n+h-1}{2}$ in case (ii). Thus from \eqref{d=2} it follows that
\begin{equation}\label{d=2 bis}
q^{h+1}N_q(Q^{n-h-1})+N_q(\mathbb{P}^h)=
\begin{cases}
q^{\frac{n+h+1}{2}}N_q(\mathbb{P}^{n-\frac{n+h+1}{2}})+N_q(\mathbb{P}^{\frac{n+h+1}{2}})  & \text{(i)\ , } \\ 
q^{\frac{n+h-1}{2}}N_q(\mathbb{P}^{n-\frac{n+h-1}{2}})+N_q(\mathbb{P}^{\frac{n+h-1}{2}})  & \text{(ii)\ . }
\end{cases} 
\end{equation}
If $Q^{n-h-1}$ is a hyperbolic quadric hypersurface, then 
$N_q(Q^{n-h-1})=\left(q^{\frac{n-h-1}{2}}+1\right)\cdot N_q(\mathbb{P}^{\frac{n-h-1}{2}})$ and \eqref{d=2 bis} becomes an identity.
So, case (i) occurs for any $h$ such that $-1\leq h\leq n-1$. On the other hand, if $Q^{n-h-1}$ is as in case (ii), then 
$N_q(Q^{n-h-1})=\left(q^{\frac{n-h-1}{2}+1}+1\right)\cdot N_q(\mathbb{P}^{\frac{n-h-1}{2}-1})$ and \eqref{d=2 bis} gives
$q^{\frac{n+h+1}{2}}+q^{\frac{n+h-1}{2}}=0$, which is clearly a numerical contradiction because $q\geq 2$ and $\frac{n+h+1}{2}>0$.
This proves case $(3)$ in Theorem \ref{thm1}.

\medskip

Finally, assume that $d=\sqrt{q}+1$. Denote by $\mathcal{H}_h$ the two statements $(a)$ and $(b)$ as in case $(2)$ of Theorem \ref{thm1}. We will prove $\mathcal{H}_h$
by induction on $h\in\mathbb{Z}_{\geq 1}$, so $h=1$ is the first step of the induction. 

If $X^2\subset\mathbb{P}^3$ is a surface of
degree $\sqrt{q}+1$ with $N_q(X^2)=\Theta_{2,k_{X^n}}^{\sqrt{q}+1,q}$, then from \cite{HK5} we know that $X^2$ is a nonsingular Hermitian surface 
in $\mathbb{P}^3$. This shows $(a)$ of case $(2)$ for $h=1$. Let $X^3\subset\mathbb{P}^4$ be a hypersurface of
degree $\sqrt{q}+1$ with $N_q(X^3)=\Theta_{3,k_{X^n}}^{\sqrt{q}+1,q}$. If $\mathrm{Sing}(X^{3})(\mathbb{F}_q)=\emptyset$, then
there exists a point $p\in X^3(\mathbb{F}_q)$ such that $X^2:=X^3\cap T_pX^3\subset\mathbb{P}^3$ is a surface in $\mathbb{P}^3$ singular at $p$.
By Lemma \ref{lemma} $(4)$ we see that $N_q(X^2)=N_q(X^3\cap T_pX^3)=\Theta_{2,k_{X^n}}^{\sqrt{q}+1,q}$ because $p$ is a nonsingular $\mathbb{F}_q$-point.
Thus, if $k_{X^3}=2$ then from \cite{T} we deduce that $X^3$ is a cone over a nonsingular Hermitian surface with vertex an $\mathbb{F}_q$-point, a contradiction.
If $k_{X^3}\leq 1$, then $k_{X^3}=1$ and since $N_q(X^2)=\Theta_{2,k_{X^n}}^{\sqrt{q}+1,q}$ by \cite{HK5} we deduce that $X^2$ is a nonsingular Hermitian surface,
which gives again a contradiction. So, suppose that $\mathrm{Sing}(X^{3})(\mathbb{F}_q)\neq\emptyset$. Then by Lemma \ref{lemma} $(3)$ we have
$X^3=p*X^2$ where $X^2$ is as in $(a)$ of $\mathcal{H}_1$. This shows that $X^3$ is a cone over a nonsingular Hermitian $\mathbb{F}_q$-surface, i.e. $(b)$ of case $(2)$ for $h=1$ is true.
This completes the proof of the statement $\mathcal{H}_h$ for $h=1$. 

Assume now that $\mathcal{H}_h$ is true for some $h\in\mathbb{Z}_{\geq 1}$.
Let $n=2(h+1)$ with $h\in\mathbb{Z}_{\geq 1}$ and $1\leq k_{X^{2(h+1)}}\leq 2(h+1)-1$. First, suppose that $\mathrm{Sing}(X^{2(h+1)})(\mathbb{F}_q)=\emptyset$
and let $p\in X^{2(h+1)}(\mathbb{F}_q)$ be a nonsingular point. Define $X^{2h+1}:=X^{2(h+1)}\cap T_pX^{2(h+1)}$ and note that $p\in\mathbb{P}^{k_{X^{2(h+1)}}}\subset X^{2h+1}$ for some $\mathbb{F}_q$-linear subspace $\mathbb{P}^{k_{X^{2(h+1)}}}\subseteq X^{2(h+1)}$ (see \cite[Proposition 5.9 (i)]{HK7}).  
Hence $k_{X^{2h+1}}=k_{X^{2(h+1)}}$. Since by Lemma \ref{lemma} $(4)$ we have $N_q(X^{2(h+1)})>N_q(X^{2h+1})$,  let $p'\notin T_pX^{2(h+1)}$ be an $\mathbb{F}_q$-point of $X^{2(h+1)}$ and consider the tangent $\mathbb{F}_q$-linear space 
$T_{p'}X^{2(h+1)}=\mathbb{P}^{2(h+1)}$. Define also $X'^{2h+1}:=X^{2(h+1)}\cap T_{p'}X^{2(h+1)}$ and observe that $k_{X'^{2h+1}}=k_{X^{2(h+1)}}$. Note that $k_{X^{2(h+1)}}\leq 2h=\dim X^{2h+1}-1$, otherwise by \cite{T} we would get 
$\mathrm{Sing}(X^{2(h+1)})(\mathbb{F}_q)\neq\emptyset$. Thus $k_{X^{2h+1}}\leq 2h$ and by Lemma \ref{lemma} $(4)$ and the induction hypothesis, we see that $X^{2h+1}=\mathbb{P}^l*X^{2h-l}$ for some
$l\in\mathbb{Z}_{\geq 0}$. If $l>0$ then the $\mathbb{F}_q$-linear subspace $\mathbb{P}^l$ intersect $X'^{2h+1}$ at least at one $\mathbb{F}_q$-point $\overline{p}$.
Since $X'^{2h+1}$ is singular at $p'$, by Lemma \ref{lemma} $(1)$, $(2)$ and $(4)$ there is an $\mathbb{F}_q$-linear subspace $\mathbb{P}^{k_{X^{2(h+1)}}}$
which contains $p'$ and $\overline{p}$. Thus there exists an $\mathbb{F}_q$-line $L:=\langle p',\overline{p}\rangle$ in $X^{2(h+1)}$.
Since $T_{\overline{p}}X^{2(h+1)}$ contains $T_{\overline{p}}X^{2h+1}=\mathbb{P}^{2(h+1)}$ because $\overline{p}\in\mathrm{Sing}(X^{2h+1})(\mathbb{F}_q)$ and the line $L\subsetneq T_{\overline{p}}X^{2h+1}$, we conclude that $T_{\overline{p}}X^{2h+1}\cup L\subseteq T_{\overline{p}}X^{2(h+1)}$, i.e.
$T_{\overline{p}}X^{2(h+1)}=\mathbb{P}^{2(h+1)+1}$. Hence $\overline{p}$ is a singular $\mathbb{F}_q$-point of $X^{2(h+1)}$, but this is a contradiction.
Therefore, $X^{2h+1}=p*X_H^{2h}$ and this gives $k_{X^{2(h+1)}}=k_{X^{2h+1}}=h+1=\frac{\dim X^{2(h+1)}}{2}$. So, by \cite[Theorem 6.3]{HK7}
we conclude that $X^{2(h+1)}$ is a nonsingular Hermitian $\mathbb{F}_q$-hypersurface. 

Suppose now that $\mathrm{Sing}(X^{2(h+1)})(\mathbb{F}_q)\neq\emptyset$. By Lemma \ref{lemma} $(3)$ we know that $X^{2(h+1)}=p*X^{2h+1}$ for some $p\in\mathrm{Sing}(X^{2(h+1)})(\mathbb{F}_q)$.
Moreover, we have $\deg X^{2h+1}=\sqrt{q}+1$, $k_{X^{2h+1}}=k_{X^{2(h+1)}}-1$ and 

$$N_q(X^{2h+1})=\frac{\Theta_{n,k_{X^{2(h+1)}}}^{\sqrt{q}+1,q}-1}{q}=\Theta_{n-1,k_{X^{2h+1}}}^{\sqrt{q}+1,q} \ .$$

\noindent Thus, by induction, we get the statement $(a)$ of $\mathcal{H}_{h+1}$. Now, let $n=2(h+1)+1$ for some $h\in\mathbb{Z}_{\geq 1}$ and $1\leq k_{X^{2(h+1)+1}}\leq 2(h+1)$. 
Suppose that $\mathrm{Sing}(X^{2(h+1)+1})(\mathbb{F}_q)=\emptyset$
and let $p\in X^{2(h+1)+1}(\mathbb{F}_q)$ be a nonsingular point. Consider $X^{2(h+1)}:=X^{2(h+1)+1}\cap T_pX^{2(h+1)+1}$ and note that $p\in\mathbb{P}^{k_{X^{2(h+1)+1}}}\subset X^{2(h+1)}$ for some $\mathbb{F}_q$-linear subspace $\mathbb{P}^{k_{X^{2(h+1)+1}}}$. 
Hence $k_{X^{2(h+1)}}=k_{X^{2(h+1)+1}}$. Moreover, since $p$ is a singular $\mathbb{F}_q$-point of $X^{2(h+1)}$,  by the part $(a)$ of $\mathcal{H}_{h+1}$ and Lemma \ref{lemma} $(4)$ we obtain that $X^{2(h+1)}=\mathbb{P}^l*X_H^{2h+1-l}$
for some $l\in\mathbb{Z}_{\geq 1}$. By considering another nonsingular $\mathbb{F}_q$-point $p'$ of $X^{2(h+1)}$ not lying on $T_pX^{2(h+1)+1}$,
by arguing as above in the previous case, we conclude that $\mathbb{P}^l\cap T_{p'}X^{2(h+1)+1}$ gives at least a
singular $\mathbb{F}_q$-point $\overline{p}$ of $X^{2(h+1)+1}$, which is a contradiction. 

Finally, assume that $\mathrm{Sing}(X^{2(h+1)+1})(\mathbb{F}_q)\neq\emptyset$.
Then from Lemma \ref{lemma} $(3)$ it follows that $X^{2(h+1)+1}=p*X^{2(h+1)}$ for some $\mathbb{F}_q$-point $p$ on $X^{2(h+1)+1}$. By applying the part $(a)$ of $\mathcal{H}_{h+1}$ to $X^{2(h+1)}$, we obtain the statement $(b)$ of $\mathcal{H}_{h+1}$. This proves that $\mathcal{H}_{h+1}$ is true whenever $\mathcal{H}_{h}$ is true. This shows that the statement $\mathcal{H}_{h}$
is true for any $h\in\mathbb{Z}_{\geq 1}$, concluding the proof of Theorem~\ref{thm1}. \hfill $\square$

\bigskip

Finally, in the nonsingular case, from Theorems \ref{thm1} and \ref{thm2} one can deduce immediately the following result.

\begin{cor}\label{corollary}
Let $X^n\subset\mathbb{P}^{n+1}$ be a nonsingular hypersurface of degree $d\geq 2$ and dimension $n\geq 1$ defined over $\mathbb{F}_q$. Then
$N_q(X^n)\leq \Theta_{n,k_{X^n}}^{d,q}$
and equality holds if and only if one of the following possibilities occurs:
\begin{enumerate}
\item[$(1)$] $n=1$, $d=2$, $k_{X^1}=0$ and $X^1$ is projectively equivalent to the plane conic
$$X_0^2+X_1^2+X_2^2=0\ ;$$
\item[$(2)$] $n=2$, $d=2$, $k_{X^2}=0$ and $X^2$ is projectively equivalent to an elliptic surface
$$\alpha X_0^2+X_0X_1+X_1^2=0$$
with $\alpha\in\{ t\in\mathbb{F}_q\ |\ t+t^2+t^4+ \dots +t^{2^{r-1}}=1 \}$ if $q=2^r$ for some $r\in\mathbb{Z}_{\geq 1}$ and 
such that $1-4\alpha $ is a non-square if $q$ is odd\ ;
\item[$(3)$] $n\geq 2$ is even, $k_{X^n}=\frac{n}{2}$ and one of the following cases holds:
\begin{enumerate}
\item[$(a)$] $d=q+1$ and $X^n$ is a space-filling hypersurface

$$(X_0,\ldots,X_{n+1})\ A\ {}^t \! (X_0^q,\ldots,X_{n+1}^q)=0,$$ 
where $A=\left( a_{ij}\right)_{i,j=1,\ldots,n+2}$ is an
$(n+2)\times (n+2)$ matrix such that ${}^t \! A=-A$, $a_{kk}=0$ for every $k=1,\ldots,n+2$ and $\det A\neq 0$\ ;
\item[$(b)$] $d=\sqrt{q}+1$ and $X^{n}$ is projectively equivalent to a 
nonsingular Hermitian hypersurface
$$X_0^{\sqrt{q}+1}+X_1^{\sqrt{q}+1}+ \dots + X_{n+1}^{\sqrt{q}+1}=0\ ;$$
\item[$(c)$] $d=2$ and $X^n$ is projectively equivalent to the hyperbolic quadric hypersurface 
$$X_0X_1+X_2X_3+ \dots + X_{n}X_{n+1}=0\ .$$
\end{enumerate}
\end{enumerate}
\end{cor}

\end{document}